\documentclass[11pt]{article}

\usepackage{amsmath,amsthm,amscd,latexsym}
\usepackage{amsfonts,pdfsync,color,graphicx}
\usepackage[psamsfonts]{amssymb}
\usepackage{epsfig}
\usepackage{color}
\usepackage[noadjust]{cite}
\usepackage{accents}

\usepackage{listings}

\usepackage[top=1.5in,bottom=1.5in,left=1.3in,right=1.0in]{geometry}

\newcommand{\N}{{\mathbb N}}
\newcommand{\R}{{\mathbb R}}

\newtheoremstyle{mystyle}               
  {}                
  {}                
  {}        
  {}                
  {\bfseries \itshape}       
  {.}      
  { }      
  {}       

\newtheorem{theorem}{Theorem}[section]

\newtheorem{lemma}[theorem]{Lemma} 
\newtheorem{corollary}[theorem]{Corollary}
\theoremstyle{definition}
\newtheorem{definition}[theorem]{Definition}
\newtheorem{example}[theorem]{Example}	
\theoremstyle{mystyle}
\newtheorem{remark}[theorem]{Remark}

\numberwithin{equation}{section}

\title{An explicit formula of the parameter dependence of de partial derivatives of the Green's functions related to arbitrary two-point boundary conditions}
\author{Alberto Cabada and Luc\'{\i}a L\'opez-Somoza\\
	CITMAga, 15782 Santiago de Compostela, Galicia, Spain.\\ Departamento de Estat\'istica, An\'alise Matem\'atica e Optimizaci\'on, \\
	Facultade de Matem\'aticas,
	Universidade de Santiago de Compostela,\\ Galicia, Spain.\\ alberto.cabada@usc.es and lucia.lopez.somoza@usc.es}

\begin{document}
\maketitle
\begin{abstract}
In this paper we obtain an explicit formula of the parameter dependence of the partial derivatives of the Green's functions related to two-point boundary conditions. Such expression follows as an integral of both kernels times the difference of the corresponding parameters of each Green's function. As a direct consequence, we deduce a simpler proof of the monotony of the constant sign of the partial derivative of a Green's function with respect to a real parameter. As a consequence, we improve the results obtained in \cite{C1}, where the monotone dependence was proved for the constant sign Green's function (not for any ot its partial derivatives) and under weaker assumptions on the Green's function. The arguments are valid for any other types of Ordinary Differential Equations coupled to Nonlocal Conditions. Moreover, analogous ideas could be developed for Partial and Fractional Differential Equations.
\end{abstract}

\section{Introduction}

Let $M\in \R$ and $k \in \{0,\ldots,n-1\}$ be given, and consider the following $n$-th order linear two-point boundary value problem:
\begin{equation*}\label{1}
\qquad \qquad \qquad T_{n,k}\left[M\right]u\left(t\right)=\sigma\left(t\right),\quad t\in I:=\left[a,b\right],\quad B_{i}\left(u\right)=0,\quad i=1,\ldots,n, \qquad \qquad\quad {(P_{n,k}^M)}
\end{equation*}
where, 
\begin{equation*}
T_{n,k}\left[M\right]u\left(t\right):=L_{n} u\left(t\right)+M\, u^{(k)}(t), \quad t \in I, 
\end{equation*}
with 
\begin{equation*}
L_{n} u\left(t\right)= u^{\left(n\right)}\left(t\right)+a_{1}\left(t\right) u^{\left(n-1\right)}\left(t\right)+\cdots +a_{n}\left(t\right) u\left(t\right),\quad t\in I,
\end{equation*}
 and
\begin{equation*}
B_{i}\left(u\right)=\displaystyle \sum_{j=0}^{n-1} \left(\alpha_{j}^{i} u^{\left(j\right)}\left(a\right)+\beta_{j}^{i} u^{\left(j\right)}\left(b\right) \right),\quad i=1,\ldots ,n,
\end{equation*}
being $\alpha_{j}^{i}$, $\beta_{j}^{i}$  real constants for all  $i\in \{1,\ldots ,n\}$ and $j\in \{0,\ldots ,n-1\}$,   $\sigma$, $a_{j} \in C(I)$ for all $j\in \{0,\ldots ,n\}$.

We are looking for solutions that belongs to the space $C^n(I)$.

 Despite we consider the previous ODE coupled to the general two-point boundary conditions, as we will see along the proofs of the paper, the arguments could be adapted for a wider set of situations that cover, among others, ODEs with Nonlocal condtions, as Difference, Fractional or Partial Differential Equations.

First, we introduce the concept of Green's function related to Problem $(P_{n,k}^M)$, see \cite[Definition 1.4.1]{C1} for details.

\begin{definition}
	\label{d-Green-function}
 We say that $g_{n,k}[M] \in C^{n-2}(I \times I) \cap C^{n}((I \times I)\backslash \{(t,t), t \in I\})$ is the Green's function related to Problem $(P_{n,k}^M)$ if and only if it is the unique solution of  problem 
\begin{equation}
	\label{e-Tk}
	T_{n,k}\left[M\right] \left(g_{n,k}[M](t,s)\right)=0,\quad t\in I\backslash\{s\},\quad B_{i}\left(g_{n,k}[M](\cdot,s)\right)=0,\quad i=1,\ldots ,n,
\end{equation}
for any $ s \in (a,b)$ fixed.

Moreover, it satisfies the following jump condition at the diagonal of the square of definition:

For each $t\in (a,b)$ there exist and are finite, the lateral limits 
\begin{equation}
\label{condition-limits}
\hspace{-.4cm} \frac{\partial^{n-1}}{\partial t^{n-1}}g_{n,k}[M](t^{-},t)=\frac{\partial^{n-1} }{\partial t^{n-1}}g_{n,k}[M](t,t^{+}) \;\; \text{and} \;\;   \frac{\partial^{n-1}}{\partial t^{n-1}}g_{n,k}[M](t,t^{-})=\frac{\partial^{n-1}}{\partial t^{n-1}}g_{n,k}[M](t^{+},t)
\end{equation}
 and, moreover, 
\begin{equation}
	\label{condition-jump}
\hspace{-.2cm}	\frac{\partial^{n-1}}{\partial t^{n-1}}g_{n,k}[M](t^{+}, t)-\frac{\partial^{n-1} }{\partial t^{n-1}}g_{n,k}[M](t^{-}, t)=\frac{\partial^{n-1}}{\partial t^{n-1}}g_{n,k}[M](t, t^{-})-\frac{\partial^{n-1}}{\partial t^{n-1}}g_{n,k}[M](t,t^{+})=1.
	\end{equation}
\end{definition}
 
 Notice that, as a direct consequence of the regularity assumptions of the Green's function coupled to the jump condition \eqref{condition-jump} in previous definition, we deduce the following result.
 
 \begin{corollary}
 	\label{c-ceros-green-1} Let $M \in \R$, $k, \, l \in \{0,\ldots,n-1\}$ be given.  	
 	Then, for any $s \in (a,b)$ fixed, the following properties are fulfilled:
 	\begin{enumerate}
 		\item[$(i)$] If $\displaystyle\frac{\partial^l}{\partial t^l} g_{n,k}[M]\left(\cdot,s\right)$ vanishes on  $[a,s)$ then it cannot be identically zero on $(s,b]$.
 		\item[$(ii)$] If $\displaystyle\frac{\partial^l}{\partial t^l} g_{n,k}[M]\left(\cdot,s\right)$ vanishes on  $(s,b]$ then it cannot be identically zero on $[a,s)$.
 	\end{enumerate}
 \end{corollary}

In order to deduce some comparison results for the constant sign of the partial derivatives of the Green's functions, we prove the following preliminary technical results.

\begin{lemma}
	\label{l-regularity-Green-function-l}
	Let $k \in \{0, \ldots,n-1\}$, $l \in \N$ and $M \in \R$ be fixed. Suppose that $a_j \in C^l(I)$ for all $j\in \{1, \ldots,n\}$, then $g_{n,k}[M] \in C^{n-2}(I \times I) \cap C^{n+l}((I \times I)\backslash \{(t,t), t \in I\})$. 
	
	Moreover, for each $t\in (a,b)$ and $m \in \{n-1,\ldots,n+l\}$, there exist and are finite, the lateral limits 
	\begin{equation}
		\label{e-limit-l}
		\frac{\partial^{m}}{\partial t^{m}}g_{n,k}[M](t^{-},t)=\frac{\partial^{m} }{\partial t^{m}}g_{n,k}[M](t,t^{+}) \quad \text{and} \quad   \frac{\partial^{m}}{\partial t^{m}}g_{n,k}[M](t,t^{-})=\frac{\partial^{m}}{\partial t^{m}}g_{n,k}[M](t^{+},t).
	\end{equation}
\end{lemma}

\begin{proof}
	From equation \eqref{e-Tk}, since $a_j \in C^l(I)$ for all $j\in \{1, \ldots,n\}$ and $g_{n,k}[M] \in C^{n}((I \times I)\backslash \{(t,t), t \in I\})$, by defining $p(j):=\min{\{j,l\}}$, we deduce that each term on \eqref{e-Tk} satisfies
	\[
a_j(t) \, 	\frac{\partial^{n-j}}{\partial t^{n-j}}g_{n,k}[M](t,s)  \in C^{p(j)}((I \times I)\backslash \{(t,t), t \in I\}), \quad j\in \{1, \ldots,n\}.
	\]
Thus, we have that, provided $l \ge 1$, $g_{n,k}[M] \in C^{n+1}((I \times I)\backslash \{(t,t), t \in I\})$. 

But this property implies that
	\[
a_j(t) \, 	\frac{\partial^{n-j}}{\partial t^{n-j}}g_{n,k}[M](t,s)  \in C^{p(j+1)}((I \times I)\backslash \{(t,t), t \in I\}), \quad j\in \{1, \ldots,n\}
\]
and, in consequence,   $g_{n,k}[M] \in C^{n+2}((I \times I)\backslash \{(t,t), t \in I\})$, whenever $l \ge 2$.

By recurrence, using that $p(m)=l$ for all $m \ge l$, we arrive at
\[
g_{n,k}[M] \in C^{n+l}((I \times I)\backslash \{(t,t), t \in I\}).
\]

Arguing in an analogous way on the limit identities \eqref{condition-limits}, we deduce the equalities in \eqref{e-limit-l} from \eqref{e-Tk}. 
	\end{proof}

\begin{remark}
	Notice that, under the hypotheses of Lemma \ref{l-regularity-Green-function-l}, we have that for any  $m \in \{n-1,\ldots,n+l\}$, the partial derivatives $\frac{\partial^{m}}{\partial t^{m}}g_{n,k}[M]$ can be continuously extended to the closed triangles $\{(t,s), a\le t \le s\le b\}$ and  $\{(t,s), a\le s\le t \le b\}$. In particular, by considering such continuous extensions, the partial derivatives $\frac{\partial^{j}}{\partial t^{j}}g_{n,k}[M]$ are bounded in $I \times I$ for all $j \in \{0,\ldots,n+l\}$.
\end{remark}

\begin{remark}
	\label{r-ceros-u'}
It is very well know that any nontrivial solution of the homogeneous linear equation $T_n[M](u(t))=0$, $t \in I$, has (if any) a finite number of zeros on the bounded interval $I$. Despite this, provided that $a_j \in C^l(I)$ for all $j\in \{1, \ldots,n\}$, any of its derivatives $u^{(i)}$, $i\in \{1, \ldots,n+l\}$, may be non trivial and have infinitely many zeros on the bounded interval $I$. To see this, it is enough to consider the following first order homogeneous linear problem:
\begin{equation}
	\label{e-infinitos-ceros}
u'(t)=p_l(t) \, u(t), \; t \in [-1,1],\quad u(0)=1, 
\end{equation}
where, for any $l \in \N$, function $p_l$ is defined as
\[
 p_l(t):=	\begin{cases}
 	  t^{2\,l +1} \, \sin\left(\frac{1}{t}\right),  & \mbox{if}\; t \neq 0 \\
 	0, & \mbox{if}\; t=0.
 \end{cases}
\]  

Since $p_l \in C^{l}([-1,1])$, we have that the unique solution of Problem \eqref{e-infinitos-ceros}, which is given by the expression
\[
u(t)=e^{\int_0^t{p_l(s)\, ds}}, \quad t \in [-1,1],
\]
belongs to $C^{l+1}([-1,1])$ and is strictly positive on $[-1,1]$. 

Moreover, since function $p_l$ vanishes at infinitely many points in any neighborhood of $t=0$, we conclude, from \eqref{e-infinitos-ceros}, that $u'$ do too. So, as a direct consequence, any derivative $u^{(i)}$, $i\in \{1, \ldots, l+1\}$, is a non trivial function and has infinitely many zeros on $[-1,1]$.
\end{remark}

Tacking into account previous remark, we introduce the following result that ensures that the corresponding derivatives with respect to $t$ of the considered Green's function are the solutions of a homogeneous linear Ordinary Differential Equation and, in particular, if they are nontrivial, have a finite number of zeros on any bounded interval.
\begin{lemma}
	\label{l-Green-function-l}
	Let $k, l \in \{0, \ldots,n-1\}$ and $M \in \R$ be fixed. Suppose that $a_j \in C^l(I)$ for all $j\in \{1, \ldots,n\}$, and consider the following functions given by recurrence:
	\begin{enumerate}
		\item[(i)]  $ a_0(t) = 1 \; \mbox{for all $t \in I$}$.
		\item[(ii)] $b_{j,0}(t)=a_j(t) \; \mbox{for all $t \in I$ and } \; j \in \{0, \ldots,n\}\backslash\{k\}.$
		\item[(iii)] $b_{k,0}(t) = a_k(t)+M \; \mbox{for all $t \in I$}$.
		\item[(iv)] $b_{0,r}(t) = 1 \; \mbox{for all $t \in I$ and}\;  r \in \{1,\ldots,l\}.$
		\item[(v)] If $b_{n,r-1}(t)=0$ \; \mbox{for all $t \in I$}, then $b_{j,r}(t)=b_{j,r-1}(t)+b'_{j-1,r-1}(t)$, \;\mbox{for all $t \in I$, $j \in \{1, \ldots,n\}$}  and \mbox{$r \in \{1,\ldots,l\}$}.
		\item[(vi)] If $b_{n,r-1}(t)\neq 0$ \; \mbox{for all $t \in I$}, then $b_{j,r}(t)=b_{j,r-1}(t)+b_{n,r-1}(t) \left(\frac{b_{j-1,r-1}}{b_{n,r-1}}\right)'(t)$, \;\mbox{for all $t \in I$,} \mbox{$j \in \{1, \ldots,n\}$ and $r \in \{1,\ldots,l\}$}.
	\end{enumerate}		
Thus, if either $l=0$; or $l \in \{1, \ldots,n-1\}$ and, for any $r \in \{1,\ldots,l\}$ given, one of the two following properties holds:
	\begin{enumerate}
		\item $b_{n,r-1}\left(t\right)= 0$ for all $t \in I$,
		\item $b_{n,r-1}\left(t\right) \neq  0$ for all $t \in I$ and, either $\left(\frac{b_{n-1,r-1}}{b_{n,r-1}}\right)'(t) \neq -1$ for all $t \in I$ or $\left(\frac{b_{n-1,r-1}}{b_{n,r-1}}\right)'(t) = -1$ for all $t \in I$,
	\end{enumerate}
then, for any $s \in (a,b)$ fixed, it is satisfied that 
	\begin{equation}
		\label{e-Hl}
		H_{l} \left(\frac{\partial^l}{\partial t^l}g_{n,k}[M](t,s)\right)=0,\quad t\in I\backslash\{s\},
	\end{equation}
where
	\begin{equation*}
	H_{l} \left(u(t)\right):=u^{\left(n\right)}\left(t\right)+b_{1,l}\left(t\right) u^{\left(n-1\right)}\left(t\right)+\cdots +b_{n,l}\left(t\right) u\left(t\right),\quad t\in I.
\end{equation*}

\end{lemma}

\begin{proof}
	To avoid tedious notation, the proof will be done for operator $L_n$, i.e., the case $M=0$. 
	The general situation, for any $M \in \R$, holds by using the notation $a_k(t)\equiv a_k(t)+M$.

First, note that, from Lemma \ref{l-regularity-Green-function-l}, we know that $g_{n,k}[0] \in C^{n+l}((I \times I)\backslash \{(t,t), t \in I\}$ and, as a consequence, the left side in equation \eqref{e-Hl} is well defined.
	
	The case $l=0$ is just equation \eqref{e-Tk}. 
	
	Consider now $l=1$. For every $s \in (a,b)$ fixed, denoting by $v_s(t):=g_{n,k}[0](t,s)$, and differentiating on equation \eqref{e-Tk} with respect to $t$, we have that
	\begin{equation*}
		L_n \left(v_s'(t)\right) +a'_{1}\left(t\right) v_s^{\left(n-1\right)}\left(t\right)+\cdots +a'_{n-1}\left(t\right) v_s'\left(t\right)+a'_{n}\left(t\right) v_s\left(t\right) =0,\quad t\in I\backslash\{s\}.
	\end{equation*}

	So, if $a_{n}\left(t\right) (=b_{n,0}(t))=0$ for all $t \in I$, from $(i)-(vi)$ we deduce that \eqref{e-Hl} holds. 
	
In the second case, that is, $a_{n}\left(t\right) (=b_{n,0}(t))\neq 0$ for all $t \in I$, by substituting in \eqref{e-Tk}, we have that
	\begin{equation*}
		v_s\left(t\right)= \left( -v_s^{\left(n\right)}\left(t\right)-a_{1}\left(t\right) v_s^{\left(n-1\right)}\left(t\right)- \cdots -a_{n-1}\left(t\right) v_s'\left(t\right)\right)/a_n(t),\quad t\in I\backslash\{s\}.
	\end{equation*}

	In consequence, it is not difficult to verify that
	\begin{equation*}
		H_{1} \left(v_s'(t)\right):=v_s^{\left(n+1\right)}\left(t\right)+b_{1,1}\left(t\right) v_s^{\left(n\right)}\left(t\right)+\cdots +b_{n,1}\left(t\right) v_s'\left(t\right)=0,\quad t\in I\backslash\{s\},
	\end{equation*}
	where the coefficients are given in $(i)-(vi)$.
	
	To study the case $l=2$, it is enough to change in previous argument $a_j$ by $b_{j,1}$ for all $j \in \{0, 1, \ldots,n\}$. 
	
	The rest of the situations, for $l \le n-1$, are proven analogously.
	
	We point out that, to apply the recurrence method, we must take into account that the inequality $\left(\frac{b_{n-1,r-1}}{b_{n,r-1}}\right)'(t) \neq -1$ for all $t \in I$ implies that $b_{n,r}(t)\neq 0$ for all $t \in I$, while $\left(\frac{b_{n-1,r-1}}{b_{n,r-1}}\right)'(t) = -1$ for all $t \in I$ says us that $b_{n,r}(t)= 0$ for all $t \in I$.
\end{proof}

\begin{remark}
	\label{r-bjr}
Notice that, if $l \ge 1$, from $(i)-(vi)$ in previous lemma, it is immediate to verify that $b_{j,r} \in C^{l-r}(I)$, for all  $j \in \{1, \ldots,n\}$ and $r \in \{1,\ldots,l\}$. In particular, all the coefficients of $H_l$ are continuous in $I$. 
\end{remark}

\begin{remark}
	\label{r-bjr-2}
If the coefficients of operator $L_n$ are constant, then it is immediate to verify that Lemma \ref{l-Green-function-l} hols. Moreover, $b_{j,r}=a_j$ for all $j \in \{1, \ldots,n\}$ and $r \in \{1,\ldots,l\}$ and, as a direct consequence, $H_l=T_{n,k}[M]$ for all $l \in \{0, \ldots,n-1\}$.
\end{remark}
Thus, as a direct consequence of Lemma \ref{l-Green-function-l} and Remark \ref{r-bjr}, we arrive at the following result.

\begin{corollary}
	\label{c-ceros-green} Let $M \in \R$ and $l, k \in \{0,\ldots,n-1\}$ be given. Assume that the hypotheses in Lemma \ref{l-Green-function-l} are satisfied. Then, for any $s \in (a,b)$ fixed, the following properties are fulfilled:
	\begin{enumerate}
		\item[$(I)$] $\displaystyle\frac{\partial^l}{\partial t^l} g_{n,k}[M]\left(\cdot,s\right)$ either vanishes in $[a,s)$ or has (if any) a finite number of zeros in $[a,s)$.
		\item[$(II)$] $\displaystyle\frac{\partial^l}{\partial t^l} g_{n,k}[M]\left(\cdot,s\right)$ either vanishes on  $(s,b]$ or has (if any) a finite number of zeros in $(s,b]$.
		\item[$(III)$] If $\displaystyle\frac{\partial^l}{\partial t^l} g_{n,k}[M]\left(\cdot,s\right)$ vanishes on  $[a,s)$ then has (if any) a finite number of zeros in $(s,b]$.
		\item[$(IV)$] If $\displaystyle\frac{\partial^l}{\partial t^l} g_{n,k}[M]\left(\cdot,s\right)$ vanishes on  $(s,b]$ then has (if any) a finite number of zeros in $[a,s)$.
	\end{enumerate}
\end{corollary}

\begin{proof}
	The properties $(I)$ and $(II)$ hold from the fact, proven in  Lemma \ref{l-Green-function-l}, that $\displaystyle\frac{\partial^l}{\partial t^l} g_{n,k}[M]\left(t,s\right)$ is the solution of a homogeneous linear ordinary differential equation on $I\backslash\{s\}$ with continuous coefficients.
	
	Assertions $(III)$ and $(IV)$ hold from condition \eqref{condition-jump} and items $(I)$ and $(II)$.
\end{proof}

Now, we enunciate the following consequence of Fredholm's Alternative that ensures the existence and uniqueness of the solution of problem $(P_{n,k}^M)$.
 \begin{theorem}
 	\label{t-exis-green}
 	Problem $(P_{n,k}^M)$ 	has a unique solution if and only if the homogeneous problem 
 	\begin{equation}
 	T_{n,k}\left[M\right] u\left(t\right)=0,\quad t\in I,\quad B_{i}\left(u\right)=0,\quad i=1,\ldots ,n,
 	\end{equation}
 	has only the trivial solution.
 	
 	In such a case, the unique solution $u_{n,k}[M]$ of $(P_{n,k}^M)$ is given by the expression
 	\begin{equation*}
 		u_{n,k}[M]\left(t\right)=\displaystyle \int_{a}^{b} g_{n,k}[M]\left(t,s
 		\right) \sigma\left(s\right) ds, \quad t \in I,
 	\end{equation*}
 with $g_{n,k}[M]$ the Green's function introduced in Definition \ref{d-Green-function}.
 \end{theorem}

\section{A linking formula}

In this section we obtain a formula that links the expression of $\frac{\partial^l}{\partial t^l}	g_{n,k}[M]$, $l \in \{0, \ldots,n\}$, for two different values of the real parameter $M$. Such expression will allow us to prove the monotone behavior of such derivatives with respect to the parameter $M$, under suitable assumptions on the constant sign of such derivative at one value of the parameter and the one of $\frac{\partial^k}{\partial t^k}	g_{n,k}[M]$ at the other. 

 Before present the obtained formula, we enunciate the following result, proved by the author in \cite[Theorems 1.8.1, 1.8.5, 1.8.6 and 1.8.9]{C1}, in which is proved, for $k=0$, that the interval where the Green's functions have constant sign is, if not empty, and interval. Moreover, the Green's function is monotone nonincreasing with respect to the parameter $M$ in such interval. The proof in such reference is completely different to the one that is proved here. The proof given there follows from monotone iterative techniques and lower and upper solutions methods. The result is the following:

\begin{theorem}
	\label{t-Green-dec}
	Let $k=0$ and $M_1$, $M_2 \in \R$ be given. Suppose that Problem $(P^{M_j}_{n,0})$ has a unique solution for any continuous function $\sigma$ and $j=1,\, 2$.
Suppose that the two related Green's functions, $g_{n,0}[M_{j}]$, have the same constant sign on $I \times I$ for $j=1,2$. Then, if $M_1 < M_2$ it is satisfied that  $g_{n,0}[M_{2}]\le g_{n,0}[M_{1}]$  on $I \times I$.\\
	Moreover, for any $\bar{M} \in (M_1,M_2)$, Problem $(P^{\bar{M}}_{n,0})$ has a unique solution for any continuous function $\sigma$,  and the related Green's function $g_{n,0}[\bar{M}]$ satisfies $g_{n,0}[M_{2}]\le g_{n,0}[\bar{M}]\le g_{n,0}[M_{1}]$  on $I \times I$.
\end{theorem}

In the sequel, we deduce two explicit formulas that links the values of two Green's functions related to two different parameters. 

\begin{theorem}
	\label{t-formula-M}
	Let $M_{0}, M_{1}\in \mathbb{R}$ and $k \in \{0,\ldots,n-1\}$ be fixed. Suppose that  problem $(P^{M_j}_{n,k})$, $j=0,1$, has a unique solution. Let $g_{n,k}[M_j]$, $j=0,1$, be the Green's functions related to problem $(P^{M_j}_{n,k})$, $j=0,1$, respectively. Then, the two following identities are fulfilled for all $t,\,s \in I$:
\begin{equation}\label{e-green-M0-M1}
	g_{n,k}[M_{1}]\left(t,s\right)=g_{n,k}[M_{0}]\left(t,s\right)+\left(M_{0}-M_{1}\right) \displaystyle \int_{a}^{b} g_{n,k}[M_{0}]\left(t,r\right) \frac{\partial^k}{\partial t^k} g_{n,k}[M_{1}]\left(r,s\right) dr .
\end{equation}
\begin{equation}\label{e-green-M0-M1-2}
	g_{n,k}[M_{1}]\left(t,s\right)=g_{n,k}[M_{0}]\left(t,s\right)+\left(M_{0}-M_{1}\right) \displaystyle \int_{a}^{b}  g_{n,k}[M_{1}]\left(t,r\right) \frac{\partial^k}{\partial t^k} g_{n,k}[M_{0}]\left(r,s\right)  dr.
\end{equation}
\end{theorem}

\begin{proof}
 For any $j=0,1$, let  $u_{n,k}[M_j]$ be the unique solution of problem $(P^{M_j}_{n,k})$, $j=0,1$. From Theorem \ref{t-exis-green}, they are given by the expression 
\begin{equation}
	\label{e-uMj}
	u_{n,k}[M_j]\left(t\right)=\displaystyle \int_{a}^{b} g_{n,k}[M_{j}]\left(t,s\right) \sigma\left(s\right) ds, \; t \in I, \; \; j=0,1.
\end{equation}

To prove equality \eqref{e-green-M0-M1}, we use that problem $(P^{M_1}_{n,k})$ can be rewritten as:
\begin{equation*}
		T_k\left[M_{0}\right] u_{n,k}[M_1]\left(t\right)=\left(M_{0}-M_{1}\right) u^{(k)}_{n,k}[M_1]\left(t\right)+\sigma\left(t\right),\quad 	B_{i}\left(u_{n,k}[M_1]\right)=0,\;\; i=1,\ldots,n.
\end{equation*} 
Then, we  know that the unique solution of the previous problem  satisfies 
\begin{equation*}
	\begin{aligned}
		u_{n,k}[M_1]\left(t\right)=&\displaystyle \int_{a}^{b} g_{n,k}[M_{0}]\left(t,s\right) \left(M_{0}-M_{1}\right)\left(\displaystyle \int_{a}^{b} \frac{\partial^k}{\partial t^k}g_{n,k}[M_{1}]\left(s,r\right) \sigma\left(r\right) dr\right) ds\\
		&+\displaystyle \int_{a}^{b} g_{n,k}[M_{0}]\left(t,s\right) \sigma\left(s\right)ds\\
		=& \left(M_{0}-M_{1}\right) \displaystyle \int_{a}^{b}\left[ \int_{a}^{b}
		g_{n,k}[M_{0}]\left(t,s\right) \frac{\partial^k}{\partial t^k} g_{n,k}[M_{1}]\left(s,r\right) ds\right] \sigma\left(r\right) dr\\
		&+\displaystyle \int_{a}^{b} g_{n,k}[M_{0}]\left(t,s\right) \sigma\left(s\right)ds.
	\end{aligned}
\end{equation*}

Since previous equality is valid for all $\sigma$ in $C(I)$, using \eqref{e-uMj}, we conclude that equality \eqref{e-green-M0-M1}  holds.

We point out that the identity
\[
u^{(k)}_{n,k}[M_1](t)=\int_{a}^{b} \frac{\partial^k}{\partial t^k}g_{n,k}[M_{1}]\left(t,s\right) \sigma\left(s\right) ds
\]
follows directly from the fact that $g_{n,k}[M_1] \in C^{n-2}(I \times I)$ provided that $k \in \{0, \ldots,n-2\}$. When $k=n-1$ it holds from the fact that $g_{n,k}[M_1] \in   C^{n}((I \times I)\backslash \{(t,t), t \in I\})$ together the jump conditions \eqref{condition-limits} and \eqref{condition-jump}.  

Identity \eqref{e-green-M0-M1-2} is proven in an analogous way.
\end{proof}

Now, from Definition \ref{d-Green-function}, by a direct derivation on identities \eqref{e-green-M0-M1} and \eqref{e-green-M0-M1-2}, we arrive at the following result. 

\begin{theorem}
	\label{t-formula-M-derivatives}
	Let $l,\; k \in \{0, \ldots,n-1\}$ be given. Under the hypotheses of Theorem \ref{t-formula-M}, the following identities hold for all $t$, $ s \in I$ whenever $l<n-1$. Otherwise, if $l=n-1$, they are fulfilled for all $t$, $ s \in I$ such that $t \neq s$:
	\begin{equation}\label{e-green-M0-M1-derivative-1}
	\frac{\partial^l}{\partial t^l}	g_{n,k}[M_{1}]\left(t,s\right)=\frac{\partial^l}{\partial t^l}g_{n,k}[M_{0}]\left(t,s\right)+\left(M_{0}-M_{1}\right) \displaystyle \int_{a}^{b}  \frac{\partial^l}{\partial t^l} g_{n,k}[M_{0}]\left(t,r\right) \frac{\partial^k}{\partial t^k} g_{n,k}[M_{1}]\left(r,s\right) dr .
	\end{equation}
and
	\begin{equation}\label{e-green-M0-M1-derivative-2}
	\frac{\partial^l}{\partial t^l}	g_{n,k}[M_{1}]\left(t,s\right)=\frac{\partial^l}{\partial t^l}g_{n,k}[M_{0}]\left(t,s\right)+\left(M_{0}-M_{1}\right) \displaystyle \int_{a}^{b} \frac{\partial^l}{\partial t^l} g_{n,k}[M_{1}]\left(t,r\right) \frac{\partial^k}{\partial t^k} g_{n,k}[M_{0}]\left(r,s\right)  dr.
	\end{equation}

Moreover, for all $k \in \{0, \ldots,n-1\}$ fixed, we have the following identity for all $t$, $ s \in I$ such that $t \neq s$:
	\begin{eqnarray}
	\frac{\partial^n}{\partial t^n}	g_{n,k}[M_{1}]\left(t,s\right)&=&\frac{\partial^n}{\partial t^n}g_{n,k}[M_{0}]\left(t,s\right)+\left(M_{0}-M_{1}\right) \displaystyle \int_{a}^{b}  \frac{\partial^n}{\partial t^n} g_{n,k}[M_{0}]\left(t,r\right) \frac{\partial^k}{\partial t^k} g_{n,k}[M_{1}]\left(r,s\right) dr \nonumber \\
	&&+ \left(M_{0}-M_{1}\right)\,\frac{\partial^k}{\partial t^k} g_{n,k}[M_{1}]\left(t,s\right).\label{e-green-M0-M1-derivative-1-l=n}
\end{eqnarray}
and
\begin{eqnarray}
	\frac{\partial^n}{\partial t^n}	g_{n,k}[M_{1}]\left(t,s\right)&=&\frac{\partial^n}{\partial t^n}g_{n,k}[M_{0}]\left(t,s\right)+\left(M_{0}-M_{1}\right) \displaystyle \int_{a}^{b} \frac{\partial^n}{\partial t^n} g_{n,k}[M_{1}]\left(t,r\right) \frac{\partial^k}{\partial t^k} g_{n,k}[M_{0}]\left(r,s\right)  dr \nonumber\\
	&& + \left(M_{0}-M_{1}\right)\,\frac{\partial^k}{\partial t^k} g_{n,k}[M_{0}]\left(t,s\right).\label{e-green-M0-M1-derivative-2-l=n}
\end{eqnarray}

\end{theorem}
Thus, by identifying the equalities \eqref{e-green-M0-M1-derivative-1} and \eqref{e-green-M0-M1-derivative-2}, together \eqref{e-green-M0-M1-derivative-1-l=n} and \eqref{e-green-M0-M1-derivative-2-l=n},  we deduce the following corollary:

\begin{corollary}
	\label{t-corollary-M1-M2}
	Under the hypotheses of Theorem \ref{t-formula-M}, the following identity holds for all $k, \; l \in \{0,\ldots,n-1\}$ and all $t,\,s \in I$ (for $l=n-1$ the identity holds by subtracting the lateral limits on the diagonal of $I \times I$):
	\begin{equation}
		\label{e-integral-1-2}
		\displaystyle \int_{a}^{b} \frac{\partial^l}{\partial t^l} g_{n,k}[M_{0}]\left(t,r\right) \frac{\partial^k}{\partial t^k} g_{n,k}[M_{1}]\left(r,s\right) dr = \displaystyle \int_{a}^{b} \frac{\partial^l}{\partial t^l} g_{n,k}[M_{1}]\left(t,r\right) \frac{\partial^k}{\partial t^k} g_{n,k}[M_{0}]\left(r,s\right)  dr .
	\end{equation}
Moreover, for all $t$, $s \in I$, $t \neq s$, the following equality is fulfilled:
		\begin{eqnarray}
		\label{e-integral-1-2-n}
		\displaystyle \int_{a}^{b} \frac{\partial^n}{\partial t^n} g_{n,k}[M_{0}]\left(t,r\right) \frac{\partial^k}{\partial t^k} g_{n,k}[M_{1}]\left(r,s\right) dr +\frac{\partial^k}{\partial t^k} g_{n,k}[M_{1}]\left(t,s\right)\\
		 = \displaystyle \int_{a}^{b} \frac{\partial^n}{\partial t^n} g_{n,k}[M_{1}]\left(t,r\right) \frac{\partial^k}{\partial t^k} g_{n,k}[M_{0}]\left(r,s\right)  dr +\frac{\partial^k}{\partial t^k} g_{n,k}[M_{0}]\left(t,s\right).\nonumber
	\end{eqnarray}
\end{corollary}

As a direct consequence of Theorem \ref{t-formula-M-derivatives} we deduce the following monotony property with respect to the parameter $M$ of the $l$-th derivative of the Green's function related to problem $(P_{n,k}^M)$.

\begin{corollary}
	\label{co-signo1}
	On the conditions of Theorem \ref{t-formula-M}, given $k, \; l \in \{0,\ldots,n-1\}$ and $t,s \in I$ ($t \neq s$ if $l=n-1$), if $M_{0} \neq M_{1}$ we have that the following assertions are equivalent:
	\begin{enumerate}
		\item $\displaystyle\frac{\partial^l}{\partial t^l} g_{n,k}[M_{1}]\left(t,s\right) > \displaystyle\frac{\partial^l}{\partial t^l} g_{n,k}[M_{0}]\left(t,s\right)$.
		\item $ (M_0-M_1) \, \displaystyle \int_{a}^{b} \frac{\partial^l}{\partial t^l} g_{n,k}[M_{0}]\left(t,r\right) \frac{\partial^k}{\partial t^k} g_{n,k}[M_{1}]\left(r,s\right) dr >0$.
		\item $(M_0-M_1) \, \displaystyle \int_{a}^{b} \frac{\partial^l}{\partial t^l} g_{n,k}[M_{1}]\left(t,r\right) \frac{\partial^k}{\partial t^k} g_{n,k}[M_{0}]\left(r,s\right)  dr > 0$.	
	\end{enumerate}
	\end{corollary}

Thus, assuming the same constant sign property of the $l$-th partial derivative with respect to $t$ of the Green's function for one parameter and the $k$-th partial derivative with respect to $t$ for the other one, we are in conditions to prove the following monotony decreasing property with respect to the parameter $M$.

\begin{theorem}
	\label{t-igual-signo-gm1-gm2} 	Assume that we are under the conditions of Theorem \ref{t-formula-M} and Lemma \ref{l-Green-function-l}.  Let $k, \; l \in \{0,\ldots,n-1\}$ and $(t,s) \in (a,b) \times (a,b)$ be fixed. Suppose that $M_{0}<M_{1}$ and  that one of the two following properties is fulfilled:
	\begin{enumerate}
		\item[$(i)$] Both $\displaystyle\frac{\partial^l}{\partial t^l} g_{n,k}[M_{0}]$ and $\displaystyle\frac{\partial^k}{\partial t^k} g_{n,k}[M_{1}]$ are either non negative or non positive on $I \times I$.
		\item[$(ii)$] Both $\displaystyle\frac{\partial^l}{\partial t^l} g_{n,k}[M_{1}]$ and $\displaystyle\frac{\partial^k}{\partial t^k} g_{n,k}[M_{0}]$ are either non negative or non positive on $I \times I$.
	\end{enumerate}
Then one of the three following properties holds:
\begin{enumerate}
	\item  $\displaystyle\frac{\partial^l}{\partial t^l} g_{n,k}[M_{1}]\left(t,s\right) = \displaystyle \frac{\partial^l}{\partial t^l} g_{n,k}[M_{0}]\left(t,s\right)=0$ if $t \in [a,s)$.
		\item  $\displaystyle\frac{\partial^l}{\partial t^l} g_{n,k}[M_{1}]\left(t,s\right) = \displaystyle \frac{\partial^l}{\partial t^l} g_{n,k}[M_{0}]\left(t,s\right)=0$ if $t \in (s,b]$.
		\item 
		$\displaystyle\frac{\partial^l}{\partial t^l} g_{n,k}[M_{1}]\left(t,s\right) < \displaystyle \frac{\partial^l}{\partial t^l} g_{n,k}[M_{0}]\left(t,s\right)$.
\end{enumerate}

\end{theorem}

\begin{proof}
We will make the proof for case $(i)$. Case $(ii)$ holds in an analogous way. Moreover, we will assume that both functions are non negative on $I \times I$. The arguments when they are non positive are the same. 

To prove the result we divide the proof in several steps. All the possibilities are deduced from Corollary \ref{c-ceros-green}.
 \begin{enumerate}
	\item[$(a)$]
 First we assume that $\displaystyle\frac{\partial^k}{\partial t^k} g_{n,k}[M_{1}]\left(\cdot,s\right)>0$ on $I\backslash \{A_s\}$, with $A_s \subset I$ a finite set.
	\begin{enumerate}
\item[$(a_1)$] If $\displaystyle\frac{\partial^l}{\partial t^l} g_{n,k}[M_{0}]\left(t,\cdot\right)>0$  on $I\backslash \{B_t\}$, with $B_t \subset I$ a finite set, the strict inequality follows immediately from \eqref{e-green-M0-M1-derivative-1}.
\item[$(a_2)$] If $\displaystyle\frac{\partial^l}{\partial t^l} g_{n,k}[M_{0}]\left(t,\cdot\right)>0$  on $[a,t)\backslash \{B_t\}$, with $B_t \subset I$ a finite set and vanishes on $(t,b]$. We have that \eqref{e-green-M0-M1-derivative-1} is written as 
\[\hspace{-2cm}
\frac{\partial^l}{\partial t^l}	g_{n,k}[M_{1}]\left(t,s\right)-\frac{\partial^l}{\partial t^l}g_{n,k}[M_{0}]\left(t,s\right)=\left(M_{0}-M_{1}\right) \displaystyle \int_{a}^{t}  \frac{\partial^l}{\partial t^l} g_{n,k}[M_{0}]\left(t,r\right) \frac{\partial^k}{\partial t^k} g_{n,k}[M_{1}]\left(r,s\right) dr <0 .
\]

\item[$(a_3)$] If $\displaystyle\frac{\partial^l}{\partial t^l} g_{n,k}[M_{0}]\left(t,\cdot\right)>0$  on $(t,b]\backslash \{B_t\}$, with $B_t \subset I$ a finite set and vanishes on $[a,t)$. We have that \eqref{e-green-M0-M1-derivative-1} is written as 
\[
\hspace{-2cm}\frac{\partial^l}{\partial t^l}	g_{n,k}[M_{1}]\left(t,s\right)-\frac{\partial^l}{\partial t^l}g_{n,k}[M_{0}]\left(t,s\right)=\left(M_{0}-M_{1}\right) \displaystyle \int_{t}^{b}  \frac{\partial^l}{\partial t^l} g_{n,k}[M_{0}]\left(t,r\right) \frac{\partial^k}{\partial t^k} g_{n,k}[M_{1}]\left(r,s\right) dr <0 .
\]
\end{enumerate}
	\item[$(b)$]
Now we assume that $\displaystyle\frac{\partial^k}{\partial t^k} g_{n,k}[M_{1}]\left(\cdot,s\right)>0$  on $[a,s)\backslash \{A_s\}$, with $A_s \subset I$ a finite set and vanishes on $(s,b]$.
\begin{enumerate}
	\item[$(b_1)$] If $\displaystyle\frac{\partial^l}{\partial t^l} g_{n,k}[M_{0}]\left(t,\cdot\right)>0$  on $I\backslash \{B_t\}$, with $B_t \subset I$ a finite set, then \eqref{e-green-M0-M1-derivative-1} gives us
	\[
\hspace{-2cm}	\frac{\partial^l}{\partial t^l}	g_{n,k}[M_{1}]\left(t,s\right)-\frac{\partial^l}{\partial t^l}g_{n,k}[M_{0}]\left(t,s\right)=\left(M_{0}-M_{1}\right) \displaystyle \int_{a}^{s}  \frac{\partial^l}{\partial t^l} g_{n,k}[M_{0}]\left(t,r\right) \frac{\partial^k}{\partial t^k} g_{n,k}[M_{1}]\left(r,s\right) dr <0 .
	\]
	\item[$(b_2)$] If $\displaystyle\frac{\partial^l}{\partial t^l} g_{n,k}[M_{0}]\left(t,\cdot\right)>0$  on $[a,t)\backslash \{B_t\}$, with $B_t \subset I$ a finite set and vanishes on $(t,b]$. We have that 
	\[
\hspace{-2cm}	\frac{\partial^l}{\partial t^l}	g_{n,k}[M_{1}]\left(t,s\right)-\frac{\partial^l}{\partial t^l}g_{n,k}[M_{0}]\left(t,s\right)=\left(M_{0}-M_{1}\right) \displaystyle \int_{a}^{\min{\{t,s\}}}  \frac{\partial^l}{\partial t^l} g_{n,k}[M_{0}]\left(t,r\right) \frac{\partial^k}{\partial t^k} g_{n,k}[M_{1}]\left(r,s\right) dr <0 ,
	\]
	\item[$(b_3)$] If $\displaystyle\frac{\partial^l}{\partial t^l} g_{n,k}[M_{0}]\left(t,\cdot\right)>0$  on $(t,b]\backslash \{B_t\}$, with $B_t \subset I$ a finite set and vanishes on $[a,t)$. We have that if $t\in (a,s)$ then
	\[
\hspace{-2cm}	\frac{\partial^l}{\partial t^l}	g_{n,k}[M_{1}]\left(t,s\right)-\frac{\partial^l}{\partial t^l}g_{n,k}[M_{0}]\left(t,s\right)=\left(M_{0}-M_{1}\right) \displaystyle \int_{t}^{s}  \frac{\partial^l}{\partial t^l} g_{n,k}[M_{0}]\left(t,r\right) \frac{\partial^k}{\partial t^k} g_{n,k}[M_{1}]\left(r,s\right) dr <0 .
	\]
	On the contrary, if $t\in (s,b)$, we conclude that
	\[
	\frac{\partial^l}{\partial t^l}	g_{n,k}[M_{1}]\left(t,s\right)=\frac{\partial^l}{\partial t^l}g_{n,k}[M_{0}]\left(t,s\right)=0.
	\]
	Thus, we are in the situation {\it 2.} of the enunciate.
\end{enumerate}
	\item[$(c)$]
Now we assume that $\displaystyle\frac{\partial^k}{\partial t^k} g_{n,k}[M_{1}]\left(\cdot,s\right)>0$  on $(s,b]\backslash \{A_s\}$, with $A_s \subset I$ a finite set and vanishes on $[a,s)$.
\begin{enumerate}
	\item[$(c_1)$] If $\displaystyle\frac{\partial^l}{\partial t^l} g_{n,k}[M_{0}]\left(t,\cdot\right)>0$  on $I\backslash \{B_t\}$, with $B_t \subset I$ a finite set, then \eqref{e-green-M0-M1-derivative-1} gives us
	\[
\hspace{-2cm}	\frac{\partial^l}{\partial t^l}	g_{n,k}[M_{1}]\left(t,s\right)-\frac{\partial^l}{\partial t^l}g_{n,k}[M_{0}]\left(t,s\right)=\left(M_{0}-M_{1}\right) \displaystyle \int_{s}^{b}  \frac{\partial^l}{\partial t^l} g_{n,k}[M_{0}]\left(t,r\right) \frac{\partial^k}{\partial t^k} g_{n,k}[M_{1}]\left(r,s\right) dr <0 .
	\]
		\item[$(c_2)$] If $\displaystyle\frac{\partial^l}{\partial t^l} g_{n,k}[M_{0}]\left(t,\cdot\right)>0$  on $[a,t)\backslash \{B_t\}$, with $B_t \subset I$ a finite set and vanishes on $(t,b]$. So, if $t \in (s,b)$, we have that 
	\[
\hspace{-2cm}	\frac{\partial^l}{\partial t^l}	g_{n,k}[M_{1}]\left(t,s\right)-\frac{\partial^l}{\partial t^l}g_{n,k}[M_{0}]\left(t,s\right)=\left(M_{0}-M_{1}\right) \displaystyle \int_s^{t}  \frac{\partial^l}{\partial t^l} g_{n,k}[M_{0}]\left(t,r\right) \frac{\partial^k}{\partial t^k} g_{n,k}[M_{1}]\left(r,s\right) dr <0 .
	\]
	Now, if $t\in (a,s)$, we conclude that
	\[
	\frac{\partial^l}{\partial t^l}	g_{n,k}[M_{1}]\left(t,s\right)=\frac{\partial^l}{\partial t^l}g_{n,k}[M_{0}]\left(t,s\right)=0,
	\]
	and we are in the situation {\it 1.} of the enunciate.
	
	\item[$(c_3)$] If $\displaystyle\frac{\partial^l}{\partial t^l} g_{n,k}[M_{0}]\left(t,\cdot\right)>0$  on $(t,b]\backslash \{B_t\}$, with $B_t \subset I$ a finite set and vanishes on $[a,t)$. We deduce that
	\[
\hspace{-2cm}	\frac{\partial^l}{\partial t^l}	g_{n,k}[M_{1}]\left(t,s\right)-\frac{\partial^l}{\partial t^l}g_{n,k}[M_{0}]\left(t,s\right)=\left(M_{0}-M_{1}\right) \displaystyle \int_{\max{\{s,t\}}}^{b}  \frac{\partial^l}{\partial t^l} g_{n,k}[M_{0}]\left(t,r\right) \frac{\partial^k}{\partial t^k} g_{n,k}[M_{1}]\left(r,s\right) dr <0 .
	\]
\end{enumerate}
\end{enumerate}
\end{proof}

\begin{remark}
	\label{r-k=n-1}
	Notice that in previous proof, case $(b)$ can be  only true if $k \in \{0, \ldots,n-2\}$. The case $k=n-1$ is not possible because of the jump condition \eqref{condition-jump}.
\end{remark}
\begin{remark}
	\label{r-k=l}
	We point out that the situations $(a_2)$ and $(a_3)$ can be true only if $k \neq l$.  We prove the case $(a_2)$. The other case is deduced in a similar way. 
	
	Thus, considering the case $(a_2)$ for $k=l$, we arrive at the following contradiction:
	\[0<\frac{\partial^k}{\partial t^k}	g_{n,k}[M_{1}]\left(t,s\right)<\frac{\partial^k}{\partial t^k}g_{n,k}[M_{0}]\left(t,s\right)=0, \qquad \mbox{for all $s\in (t,b]$}.
	\]
		 Similar arguments show us that the situation $k=l$ is not possible for the cases $(b_2)$ whenever $t \in (a,s)$ and $(c_3)$ if $t \in (s,b)$
\end{remark}

\begin{remark}
We point out that the particular case of $l=k=0$ and constant sign Green's functions has been proved in  \cite[Theorems 1.8.1 and 1.8.6]{C1} but, in such proofs, the obtained inequalities are non strict. 

Furthermore, the proof of the theorem \ref{t-igual-signo-gm1-gm2}  is simpler than the one given in that reference and provides much more additional information. It obviously covers a huge set of different situations.
\end{remark}

Arguing as in the proof of Theorem \ref{t-igual-signo-gm1-gm2}  we can prove the following result for different constant sign derivatives of the Green's functions.

\begin{theorem}
	\label{t-distinto-signo-gm1-gm2} 	Assume that we are under the conditions of Theorem \ref{t-formula-M} and Lemma \ref{l-Green-function-l}.  Let $k, \; l \in \{0,\ldots,n-1\}$ and $(t,s) \in (a,b) \times (a,b)$ be fixed. Suppose that $M_{0}<M_{1}$ and  that one of the two following properties is fulfilled:
	\begin{enumerate}
		\item[$(i)$] Both $\displaystyle\frac{\partial^l}{\partial t^l} g_{n,k}[M_{0}]$ and $\displaystyle\frac{\partial^k}{\partial t^k} g_{n,k}[M_{1}]$ have constant sign on $I \times I$. Moreover, one is non negative and the other one is non positive on $I \times I$.
		\item[$(ii)$] Both $\displaystyle\frac{\partial^l}{\partial t^l} g_{n,k}[M_{1}]$ and $\displaystyle\frac{\partial^k}{\partial t^k} g_{n,k}[M_{0}]$ have constant sign on $I \times I$. Moreover, one is non negative and the other one is non positive on $I \times I$.
	\end{enumerate}
	Then one of the two following properties holds:
	\begin{enumerate}
		\item  $\displaystyle\frac{\partial^l}{\partial t^l} g_{n,k}[M_{1}]\left(t,s\right) = \displaystyle \frac{\partial^l}{\partial t^l} g_{n,k}[M_{0}]\left(t,s\right)=0$ if $t \in [a,s)$.
		\item  $\displaystyle\frac{\partial^l}{\partial t^l} g_{n,k}[M_{1}]\left(t,s\right) = \displaystyle \frac{\partial^l}{\partial t^l} g_{n,k}[M_{0}]\left(t,s\right)=0$ if $t \in (s,b]$.
		\item 
		$\displaystyle\frac{\partial^l}{\partial t^l} g_{n,k}[M_{1}]\left(t,s\right) > \displaystyle \frac{\partial^l}{\partial t^l} g_{n,k}[M_{0}]\left(t,s\right)$.
	\end{enumerate}
\end{theorem}

\begin{remark}
	\label{r-signo-r=l}
	We point out that Remark \ref{r-k=l} and the analogous one to Remark \ref{r-k=n-1}  are valid for this situation too. 
\end{remark}

As a direct consequence of previous result, we arrive at the following one when $k=l$.

\begin{corollary}
	\label{l-distinto-signo-k=l} 	Assume that we are under the conditions of Theorem \ref{t-formula-M} and Lemma \ref{l-Green-function-l}. Then, if
$\displaystyle\frac{\partial^k}{\partial t^k} g_{n,k}[M_{0}] \ge 0 \ge \displaystyle\frac{\partial^k}{\partial t^k} g_{n,k}[M_{1}]$ on $I \times I$, we have that $M_1<M_0$.
\end{corollary}

\section{Parameter Set of Constant Sign Green's Function}
In this section, under suitable additional conditions on the sign of the Green's function related to the operator $T_{n,k}\left[M\right]$, we will study the monotony behavior of their $l$-th order derivatives with respect to $t$, for $l \in \{0, \ldots,n-1\}$ and $k=0$.

In the line of \cite[Section 1.8]{C1} we introduce the concepts of strongly positive and strongly negative Green's function as follows.
\begin{definition}
	We say that the Green's function related to problem $(P_{n,k}^M)$ is strongly positive in $I \times I$ if it satisfies the following property:
	\begin{enumerate}
		\item[$(P_g)$] There is a continuous function $\phi(t) >0$ for all $t \in (a,b)$ and $k_1, \; k_2 \in L^1(I)$, such that $0<k_1(s)<k_2(s)$ for a.e. $s \in I$, satisfying
		$$ \phi(t)\, k_1(s) \le  g_{n,k}[M](t,s) \le \phi(t)\, k_2(s), \quad \mbox{for a.e. } \; (t,s)  \in I \times I.$$
	\end{enumerate}
		We say that the Green's function related to problem $(P_{n,k}^M)$ is strongly negative in $I \times I$ if it satisfies the following property:
	\begin{enumerate}
		\item[$(N_g)$] There is a continuous function $\phi(t) >0$ for all $t \in (a,b)$ and $k_1, \; k_2 \in L^1(I)$, such that $k_1(s)<k_2(s)<0$ for a.e. $s \in I$, satisfying
		$$ \phi(t)\, k_1(s) \le  g_{n,k}[M](t,s) \le \phi(t)\, k_2(s), \quad \mbox{for a.e. } \; (t,s)  \in I \times I.$$
	\end{enumerate}
\end{definition}

Now, we say that $\lambda$ is an eigenvalue of problem $(P_{n,k}^M)$, if the following problem has a nontrivial solution:
\begin{equation}
	\label{e-eigenvalue}
	T_{n,k}\left[M\right] u\left(t\right)=\lambda u\left(t\right),\quad t\in I,\quad B_{i}\left(u\right)=0,\quad i=1,\ldots,n.
\end{equation}

Thus, considering $k=0$, for any $l \in \{0,\ldots,n-1\}$ be given, we study the dependence on $M$ of the $l$-th derivative  with respect to $t$ of the Green's function. To do this, we define  the sets of values in which the corresponding partial derivatives have constant sign on $I\times I$ as follows:

\begin{equation*}
	P_{l}=\{M\in \mathbb{R}:\;\; \frac{\partial^l}{\partial t^l}g_{n,0}[M]\left(t,s\right)\geq 0\;\; \forall \left(t,s\right)\in I\times I\}
\end{equation*} 
and
\begin{center}
	$N_{l}=\{M\in \mathbb{R}:\;\; \frac{\partial^l}{\partial t^l}g_{n,0}[M]\left(t,s\right)\leq 0 \quad \forall \left(t,s\right)\in I\times I \}.$
\end{center}

On the two following results, the case $l=k=0$ is described.

\begin{lemma}\cite[Lemma 1.8.33]{C1}
	\label{40}
	Let $\bar{M}$ be fixed. If problem $(P_{n,0}^{\bar{M}})$ has a unique solution for any $\sigma \in C(I)$ and its related Green's function $g_{n,0}[\bar{M}]$ satisfies condition $\left(P_g\right)$. Then, if the set $P_{0}$ is bounded from above, it is given as the interval $P_{0}=\left(\bar{M}-\lambda_{1},\bar{M}-\bar{\mu}\right]$, with $\lambda_{1}>0$ the smallest positive eigenvalue of problem $(P_{n,0}^{\bar{M}})$,  and $\bar{\mu}\leq 0$ such that problem $(P_{n,0}^{\bar{M}-\bar{\mu}})$  has a unique solution for any $\sigma \in C(I)$ and the related nonnegative Green's function, $g_{n,0}[\bar{M}-\bar{\mu}]$, vanishes at some points of the square $I\times I$.   
\end{lemma}

\begin{lemma}\cite[Lemma 1.8.25]{C1}
		\label{40}
	Let $\bar{M}$ be fixed. If problem $(P_{n,0}^{\bar{M}})$ has a unique solution for any $\sigma \in C(I)$ and its related Green's function $g_{n,0}[\bar{M}]$ satisfies condition $\left(N_g\right)$. Then, if the set $N_{0}$ is bounded from below, it is given as the interval $N_{0}=\left[\bar{M}-\bar{\mu},\bar{M}-\lambda_{1}\right)$, with $\lambda_{1}<0$ the biggest negative eigenvalue of problem $(P_{n,0}^{\bar{M}})$,  and $\bar{\mu}\geq 0$ such that problem $(P_{n,0}^{\bar{M}-\bar{\mu}})$  has a unique solution for any $\sigma \in C(I)$ and the related nonpositive Green's function $g_{n,0}[\bar{M}-\bar{\mu}]$ vanishes at some points of the square $I\times I$.   
\end{lemma}

To finalize this introductory properties, we present the following relation between the extremes of $N_T$ and $P_T$.
\begin{theorem}\cite[Theorem 1.8.36]{C1}
	\label{t-NT-PT-pos}
	Let $\bar{M} \in \R$ be such that problem $(P_{n,0}^{\bar{M}})$ has a unique solution for any $\sigma \in C(I)$ and its related Green's function $g_{n,0}[\bar{M}]$ satisfies condition $\left(P_g\right)$. If the interval $N_T$ is nonempty then $\sup{(N_0)}=\inf{(P_0)}$.
\end{theorem}

\begin{theorem}\cite[Theorem 1.8.35]{C1}
	\label{t-NT-PT}
	Let $\bar{M} \in \R$ be such that problem $(P_{n,0}^{\bar{M}})$ has a unique solution for any $\sigma \in C(I)$ and its related Green's function $g_{n,0}[\bar{M}]$ satisfies condition $(N_g)$. If the interval $P_T$ is nonempty then $\sup{(N_0)}=\inf{(P_0)}$.
\end{theorem}

\begin{theorem}
	\label{t-dependencia-positiva}
	Assume that $g_{n,0}[\bar{M}]$ satisfies condition $\left(P_g\right)$ for some $\bar{M} \in \R$. By denoting $P_0=(\bar{M}_0,\bar{M}_1]$ if it is bounded from above, and $P_0=(\bar{M}_0,+\infty)$ otherwise, we have that the two following properties are fulfilled:
	\begin{enumerate}
		\item[$(i)$] Suppose that $P_l \neq \emptyset$ for some $l \in \{1,\ldots,n-1\}$ fixed and there is $M_1 \in P_l$ such that $M_1> \bar{M}_0$. Then $\inf{(P_l)}=\inf{(P_0)}=\bar{M}_0$. Moreover, we have that  function $\frac{\partial^l}{\partial t^l}g_{n,0}[M]\left(t,s\right)$ is nonincreasing with respect to $M \in (\bar{M}_0,\min{\{\sup{(P_0)},\sup{(P_l)}\}}]$. If both supreme are equal to $+\infty$ the previous interval is open on the right extreme.
			\item[$(ii)$] Suppose that $N_l \neq \emptyset$ for some $l \in \{1,\ldots,n-1\}$ fixed and there is $M_1 \in N_l$ such that $M_1> \bar{M}_0$. Then $\inf{(N_l)}=\inf{(P_0)}=\bar{M}_0$. Moreover, we have that  function $\frac{\partial^l}{\partial t^l}g_{n,0}[M]\left(t,s\right)$ is nondecreasing with respect to $M \in (\bar{M}_0,\min{\{\sup{(P_0)},\sup{(P_l)}\}}]$. If both supreme are equal to $+\infty$ the previous interval is open on the right extreme.
	\end{enumerate}
\end{theorem}

\begin{proof}
We will prove the case $(i)$. Case $(ii)$ holds in an analogous way.

We will assume that both sets are bounded from above. Thus, we denote $\tilde{M}_l=\sup{\{P_l\}} \in \R$. The other situations hold analogously.

So, let $M_0 \in P_0$ be such that $M_0<M_1$.  Thus, as an immediate application of identity \eqref{e-green-M0-M1-derivative-2} for $k=0$, we conclude that
\[
0\le \frac{\partial^l}{\partial t^l}g_{n,0}[M_{1}]\left(t,s\right)\le\frac{\partial^l}{\partial t^l}g_{n,0}[M_{0}]\left(t,s\right) \quad \mbox{for all $(t,s) \in I \times I$}.
\]

In particular, we have that $M_0 \in P_l$, and, as a consequence, $P_0 \cap P_l \neq \emptyset$.

Since this argument is valid for all $M_1 \in P_l$ such that $M_1>\bar{M}_0$ and $M_0 \in (\bar{M}_0,M_1) \cap P_0$, we conclude that $(\bar{M}_0,\min{\{\bar{M}_1,\tilde{M}_l\}}] \subset P_l$.

Now, taking $\bar{M}_0 <\tilde{M}_0<\tilde{M}_1<\min{\{\bar{M}_1,\tilde{M}_l\}}$, using again \eqref{e-green-M0-M1-derivative-2} for $k=0$, we arrive at
\[
0\le \frac{\partial^l}{\partial t^l}g_{n,0}[\tilde{M}_1]\left(t,s\right)\le\frac{\partial^l}{\partial t^l}g_{n,0}[\tilde{M}_0]\left(t,s\right), \quad \mbox{for all $(t,s) \in I \times I$}
\]
and we have that the $\frac{\partial^l}{\partial t^l}g_{n,0}[M]\left(t,s\right)$ is nonincreasing with respect to $M \in (\bar{M}_0,\min{\{\bar{M}_1,\tilde{M}_l\}}]$.\\

To see that $\inf{(P_l)}=\inf{(P_0)}=\bar{M}_0$, assume, on the contrary, that there is $M_2 \in P_l$ such that $M_2<\bar{M}_0$.  

First, since $g_{n,0}[\bar{M}]$ satisfies condition $(P_g)$ we have, in particular, that $g_{n,0}[\bar{M}](t,s)>0$ for a.e. $(t,s) \in I\times I$. So, arguing as in the proof of Theorem \ref{t-igual-signo-gm1-gm2}, $(a)$, for $k=l=0$, and working with equality \eqref{e-green-M0-M1-derivative-2}, we deduce that $g_{n,0}[M]$ is strictly decreasing with respect to $M \in (\bar{M}_0,\bar{M}]$ and, as a consequence, $g_{n,0}[M](t,s)>0$ for all $(t,s) \in (a,b)\times (a,b)$ on $(\bar{M}_0,\bar{M}]$.

In the same way, arguing again as in the proof of Theorem \ref{t-igual-signo-gm1-gm2}, $(a)$, for $k=0$, we conclude that  
\begin{equation}
	\label{des-gM-M<M0bar}
	\frac{\partial^l}{\partial t^l}g_{n,0}[M_{2}]\left(t,s\right)>\frac{\partial^l}{\partial t^l}g_{n,0}[M]\left(t,s\right)\ge 0, \; \mbox{for all $(t,s) \in (a,b) \times (a,b)$ and  $M \in (\bar{M}_0,\bar{M}_1]$.}
\end{equation}

On the other hand, since $\bar{M}_0$ is an eigenvalue of our problem, we have that the Green's function $g_{n,0}[\bar{M}_0]$ does not exist, see \cite[Theorem 1.2.10]{C1} for details. 

As a consequence, from the fact that $g_{n,0}[M]$ is nonincreasing on $(\bar{M}_0,\bar{M}_1]$, we have that there exists $\lim_{M \to \bar{M}_0^+}g_{n,0}[M](t,s)$ for all $(t,s) \in I \times I$, and it is equals to $+\infty$ on a  set of positive Lebesgue measure  ${\cal{A}} \in I \times I$. Indeed, if the limit is real for a.e. $(t,s) \in I \times I$, we have that problem $(P_{n,0}^{\bar{M}_0})$ has at least one solution for any $\sigma \in C(I)$. Now, using \cite[Theorem 1.2.10]{C1} again, we deduce that the solution is unique and, therefore, $\bar{M}_0$ is not an eigenvalue of our problem.

Now, integrating in $[a,b]$ with respect to $s$ in \eqref{e-green-M0-M1-derivative-2}, we have that
	\begin{eqnarray*}
-	\int_a^b{\frac{\partial^l}{\partial t^l}g_{n,0}[M_{0}]\left(t,s\right) \,ds}&=&-\int_a^b{\frac{\partial^l}{\partial t^l}	g_{n,0}[M_{2}]\left(t,s\right)\, ds}\\
	&& +\left(M_{0}-M_{2}\right) \displaystyle \int_{a}^{b} \int_{a}^{b} \frac{\partial^l}{\partial t^l} g_{n,0}[M_{2}]\left(t,r\right) g_{n,0}[M_{0}]\left(r,s\right)  dr ds\\
	&\ge &-\int_a^b{\frac{\partial^l}{\partial t^l}	g_{n,0}[M_{2}]\left(t,s\right)\, ds}\\
	&& +\left(M_{0}-M_{1}\right) \displaystyle \int_{\cal{A}} \frac{\partial^l}{\partial t^l} g_{n,0}[M_{2}]\left(t,r\right) g_{n,0}[M_{0}]\left(r,s\right)  dr ds.
\end{eqnarray*}

So, since $M_2 \in P_l$ is fixed, the strictly positive sign of the considered Green's functions in the integral implies that the limit when $M_0 \to \bar{M}_0^+$ of the right hand side of previous equation is equals to $+\infty$. But, such property says us that for $M_0>\bar{M}_0$, close enough to $\bar{M}_0$, the function $\frac{\partial^l}{\partial t^l}g_{n,0}[M_{0}]\left(t,s\right)$ takes strictly negative values at some points in $I \times I$, which contradicts the first part of the proof. 
\end{proof}

In an analogous way, one can prove the following result.
\begin{theorem}
		\label{t-dependencia-negativa}
	Assume that $g_{n,0}[\bar{M}]$ satisfies condition $\left(N_g\right)$ for some $\bar{M} \in \R$. By denoting $N_0=[\bar{M}_1,\bar{M}_0)$ if it is bounded from above, and $N_0=(-\infty,\bar{M}_0)$ otherwise, we have that the two following properties are satisfied:
	\begin{enumerate}
		\item[$(i)$] Suppose that $N_l \neq \emptyset$ for some $l \in \{1,\ldots,n-1\}$ fixed and there is $M_1 \in N_l$ such that $M_1< \bar{M}_0$. Then $\sup{(N_l)}=\sup{(N_0)}=\bar{M}_0$. Moreover, we have that  function $\frac{\partial^l}{\partial t^l}g_{n,0}[M]\left(t,s\right)$ is nonincreasing with respect to $M \in [\max{\{\inf{(N_0)},\inf{(N_l)}\}},\bar{M}_0)$. If both infimum are equal to $-\infty$ the previous interval is open on the left extreme.
		\item[$(ii)$] Suppose that $P_l \neq \emptyset$ for some $l \in \{1,\ldots,n-1\}$ fixed and there is $M_1 \in P_l$ such that $M_1< \bar{M}_0$. Then $\sup{(P_l)}=\sup{(N_0)}=\bar{M}_0$. Moreover, we have that  function $\frac{\partial^l}{\partial t^l}g_{n,0}[M]\left(t,s\right)$ is nondecreasing with respect to $M \in [\max{\{\inf{(N_0)},\inf{(N_l)}\}},\bar{M}_0)$. If both infimum are equal to $-\infty$ the previous interval is open on the left extreme.
	\end{enumerate}
\end{theorem}
 
\section{Examples}

\begin{example}
	\label{example-Mixto0}
	It is very well known \cite{CCS1} that the second order mixed problem 
	\[u''(t)+M\, u(t)= \sigma(t), \; t \in I, \quad u(0)=u'(1)=0,\]
	has a unique solution for every $M \neq (\frac{\pi}{2}+k\, \pi)^2$, $k=0,1, \ldots$. 
	
	Moreover, the related Green's $g_{M,0}$ function is strictly negative on $(0,1] \times (0,1]$ if and only if $M < \frac{\pi^2}{4}$ and changes its sign on $I \times I$ for every $M > \frac{\pi^2}{4}$, $M \neq (\frac{\pi}{2}+k\, \pi)^2$, $k=0,1, \ldots$. Moreover it is immediate to verify that $g_0$ satisfies condition $(N_g)$.
	
	As a direct consequence of Theorem \ref{t-dependencia-negativa}, $(i)$, for the particular case of $k=l=0$, we have that for any $(t,s) \in (0,1] \times (0,1]$ function $g_{M,0}(t,s)$ is strictly decreasing with respect to $M \in \left(-\infty,\frac{\pi^2}{4}\right)$.
	
	Now, from Definition \ref{d-Green-function}, we have that 
	\begin{equation}
		\label{e-derivada-mixto0}
	\frac{\partial^2}{\partial t^2} g_{M,0}\left(t,s\right)= -M\, g_{M,0}\left(t,s\right), \qquad t \in I\backslash\{s\}.
	\end{equation}
	Thus, if $M \in \left(0,\frac{\pi^2}{4}\right)$, we have that $\frac{\partial}{\partial t} g_{M,0}\left(\cdot,s\right)$ is strictly increasing on $I\backslash\{s\}$. So, from the fact that $\frac{\partial}{\partial t} g_{M,0}\left(1,s\right)=0$, together the jump condition \eqref{condition-jump}, we conclude that 
	\[
	\frac{\partial}{\partial t} g_{M,0}\left(t,s\right)<0 \qquad \mbox{for all $(t,s) \in  [0,1) \times [0,1)$ and $M \in \left(0,\frac{\pi^2}{4}\right)$.}
	\]
	Using again \ref{t-dependencia-negativa}, $(i)$, in this case for $k=0$ and $l=1$, we deduce that $\frac{\partial}{\partial t} g_{M,0}\left(t,s\right)$ is strictly decreasing with respect to $M \in \left(0,\frac{\pi^2}{4}\right)$.
	
	It is immediate to verify that if $M=0$ then $\frac{\partial}{\partial t} g_{0,0}\left(t,s\right)=-1$ if $t\in [0,s)$ and $\frac{\partial}{\partial t} g_{0,0}\left(t,s\right)=0$ if $t\in (s,1]$. Moreover, from \eqref{e-derivada-mixto0}, we deduce that $\frac{\partial}{\partial t} g_{M,0}$ is strictly decreasing on $I\backslash\{s\}$ for all $M<0$. In consequence, the sign of $g_{M,0}$ coupled to the boundary condition at $(1,s)$, allows us to ensure that $\frac{\partial}{\partial t} g_{M,0}$ changes its sign on $I \times I$ for all $M<0$
	\end{example}
	
\begin{example}
	\label{example-Mixto1}	
Consider now that the second order mixed problem 
	\[u''(t)+M\, u'(t)= \sigma(t), \; t \in I, \quad u(0)=u'(1)=0,\]
	it is very easy to verify that it has a unique solution for every $M \in \R$,  and the corresponding Green's function is given by the expression
$$ g_{M,1}(t,s)=	\begin{cases}
		\frac{1-e^{M s}}{M} & 0\leq s\leq t\leq 1 \\
		\frac{e^{M s} \left(e^{-M t}-1\right)}{M} & 0<t<s\leq 1
	\end{cases}, \mbox{if $M \neq 0$, and } g_{0,1}(t,s) =	\begin{cases}
		-s & 0\leq s\leq t\leq 1 \\
		-t & 0<t<s\leq 1.
	\end{cases}
$$
It is obvious that $g_{M,1}$ is negative on $(0,1] \times (0,1]$ for all $M \in \R$. Moreover, since 
$$ \frac{\partial }{\partial t}g_{M,1}(t,s)=\begin{cases}  	0 & 0\leq s\leq t\leq 1 \\
	-e^{M (s-t)} & 0<t<s\leq 1
\end{cases}.
$$

In this case, the expression of the Green's function is very easy to manage, and so we can deduce the monotony behavior with respect to the parameter $M$.

Anyway, by using Theorem \ref{t-igual-signo-gm1-gm2}, for $k=1$ and $l=0$, we deduce immediately that function $g_{M,1}(t,s)$ is strictly decreasing with respect to $M \in \R$ for all $(t,s) \in (0,1] \times (0,1]$.

The same result, applied to the values $k=l=1$, says us that $\frac{\partial}{\partial t} g_{M,1}=0$ for all $t \in (s,1]$ and it is strictly decreasing with respect to $M \in \R$ for all $t \in (0,s)$.

Notice that in this example, Theorem \ref{t-dependencia-negativa} is not applicable, because we are considering the dependence with respect the parameter coefficient of $u'$ instead of the one of $u$. 
\end{example}


\begin{thebibliography}{20}
	\bibitem{C1} A. Cabada, \textit{Green's Functions in theory of Ordinary Differential Equations}, in: Springer Briefs in Mathematics, 2014.
	\bibitem{CCS1} A. Cabada, J.A. Cid, L. López-Somoza, \textit{Green's functions and spectral theory for the Hill's equation}, Appl. Math. Comput. \textbf{286} (2016) 88-105.
	\bibitem{l} R. Hakl, P.J. Torres, \textit{Maximum and antimaximum principles for a second order differential operator with variable coefficients of indefinite sign}, Appl. Math. Comput. \textbf{217} (2011) 7599-7611.
	\bibitem{pp} P.J. Torres,\textit{Existence of one-signed periodic solutions of some second-order differential equations via a Krasnoselskii fixed pint theorem}, J. Differential Equations \textbf{190} (2003) 643-662.
	\bibitem{MZ} M. Zhang, \textit{Optimal conditions for maximum and anti-maximum principles of the periodic solution problem}, Bound. Value Probl. \textbf{(2010)} 410986, hhtp://dx.doi.org/10.1155/2010/410986, 26pp.
	\bibitem{habet} De Coster, C., Habets, P.: \textit{The  lower and upper solutions method for boundary value problems}. Handbook of Differential Equations, pp.69-160, Elsevier/North-Holand, Amsterdam (2004). 
	\bibitem{kar} Kythee, P: \textit{Green's function and linear differential equations}. Theory, applications, and computation. Chapman, Hall/CRC Applied Mathematics and Nonlinear Sience Series. CRC Press, Boca Raton (2001). 
	
	
	
\end{thebibliography}
\end{document}